\documentclass[12pt]{article}

\usepackage{enumerate}
\usepackage{amsthm}
\usepackage{amssymb}
\usepackage{amsmath}
\usepackage{amsfonts}
\usepackage{cancel}
\usepackage{times}
\usepackage{bbm}
\usepackage{cite}
\usepackage{color, soul}
\usepackage{hyperref}
\hypersetup{
	colorlinks   = true,
	citecolor    = black,
	linkcolor    = blue
}

\usepackage[paper=a4paper, top=2.5cm, bottom=2.5cm, left=2.5cm, right=2.5cm]{geometry}

\def\qed{\hfill $\vcenter{\hrule height .3mm
\hbox {\vrule width .3mm height 2.1mm \kern 2mm \vrule width .3mm
height 2.1mm} \hrule height .3mm}$ \bigskip}

\def \Sph{\mathbb{S}^{n-1}}
\def \RR {\mathbb R}
\def \NN {\mathbb N}
\def \EE {\mathbb E}

\def \PP {\mathbb P}
\def \eps {\varepsilon}

\newcommand\norm[1]{\left\lVert#1\right\rVert}
\newcommand\inner[2]{\langle#1,#2\rangle}
\newtheorem{theorem}{Theorem}
\newtheorem{lemma}{Lemma}

\theoremstyle{definition}

\theoremstyle{remark}
\newtheorem{remark}[theorem]{Remark}
\newtheorem*{remark*}{Remark}

\long\def\symbolfootnotetext[#1]#2{\begingroup
\def\thefootnote{\fnsymbol{footnote}}\footnotetext[#1]{#2}\endgroup}

\begin{document}
	\title{Non-asymptotic approximations of neural networks by Gaussian processes}
	\author{Ronen Eldan \thanks{Supported by a European Research Council Starting Grant (ERC StG) and by an Israel Science Foundation Grant no. 715/16.}\\Weizmann Institute \and Dan Mikulincer\\Weizmann Institute  \and Tselil Schramm\\Stanford University }
	\date{\today}

	\maketitle
	\begin{abstract}
		We study the extent to which wide neural networks may be approximated by Gaussian processes, when initialized with random weights. It is a well-established fact that as the width of a network goes to infinity, its law converges to that of a Gaussian process. We make this quantitative by establishing explicit convergence rates for the central limit theorem in an infinite-dimensional functional space, metrized with a natural transportation distance. We identify two regimes of interest; when the activation function is polynomial, its degree determines the rate of convergence, while for non-polynomial activations, the rate is governed by the smoothness of the function.
	\end{abstract}
	\section{Introduction}
In the past decade, artificial neural networks have experienced an unprecedented
renaissance. However, the current theory has yet to catch-up with the practice and cannot explain their impressive performance. Particularly intriguing is the fact that \emph{over-parameterized} models do not tend to over-fit, even when trained to zero error on the training set. Owing to this seemingly paradoxical fact, researchers have focused on understanding the infinite-width limit of neural networks. This line of research has led to many important discoveries such as the `lazy-training' regime \cite{chizat2019lazy,tzen2020mean} which is governed by the limiting `neural tangent kernel' (see \cite{jacot2018neural}), as well as the `mean-field' limit approach (see \cite{mei2018mean,mei2019mean, karakida2019universal} for some examples) to study the training dynamics and loss landscape.

The first to study the limiting distribution of a neural network at (a random) initialization was Neal \cite{neal1996bayesian}, who proved a \emph{Central Limit Theorem (CLT)} for two-layered wide neural networks. According to Neal's result, when initialized with random weights, as the width of the network goes to infinity, its law converges, in distribution, to a Gaussian process. Subsequent works have generalized this result to deeper networks and other architectures (\cite{williams1997computing,matthews2018gaussian,novak2018bayesian, garriga2018deep, hazan2015steps, yang2019scaling,yang2019tensor}). This correspondence between Gaussian processes and neural networks has proved to be highly influential and has inspired many new models (see \cite{yang2019scaling} for a thorough review of these models).

Towards supplying a theoretical framework to study \emph{real world} neural networks, one important challenge is to understand the extent to which existing asymptotic results, which essentially only apply to infinite networks, may also be applied to finite ones. While there have been several works in this direction (c.f. \cite{arora2019exact, geiger2020scaling, hanin2019universal,andreassen2020asymptotics, yaida2019non, antognini2019finite, aitken2020asymptotics}), to the best of our knowledge, all known results consider finite-dimensional marginals of the random process and the question of finding a finite-width quantitative analog to Neal's CLT, which applies in a functional space, has remained open. The main goal of this paper is to tackle this question.

In essence, we prove a quantitative CLT in the space of functions. On a first glance, this is a completely different setting than the classical CLT, even in high-dimensional regimes. The function space is infinite-dimensional, while all quantitative bounds of CLT deteriorate with the dimension (\cite{bobkov18bounds,courtade2017existence,bonis2020}). However, by exploiting the special structure of neural networks we are able to reduce the problem to finite-dimensional sets, where we capitalize on recent advances made in understanding the rate of convergence of the high-dimensional CLT. In particular, we give quantitative bounds, depending on the network's width and the dimension of the input, which show that, when initialized randomly, wide but finite networks can be well-approximated by a Gaussian process. The functional nature of our results essentially means that when considering the joint distribution attained on a finite set of inputs to the function, ours bounds do not deteriorate as the number of input points increases. \\

\noindent Roughly speaking, we prove the following results:
\begin{itemize}
	\item We first consider two-layered networks with polynomial activation functions. By embedding the network into a high-dimensional tensor space we prove a quantitative CLT, with a polynomial rate of convergence in a strong transportation metric.
	\item We next consider general activations and show that under a (very mild) integrability assumption, one can reduce this case to the polynomial case. This is done at the cost of weakening the transportation distance. The rate of convergence depends on the smoothness of the activation and is typically sub-polynomial.
\end{itemize}
		\paragraph{Organization:}The rest of the paper is devoted to describing and proving these results. In Section \ref{sec:background} we give the necessary background concerning random initializations of neural networks and we introduce a metric between random processes on the sphere. Our main results are stated in Section \ref{sec:results}. In Section \ref{sec:polynomialprocs} we prove results which concern polynomial activations, while in Section \ref{sec:generalactivs} we consider general activations. 
	
	\section{Background} \label{sec:background}
	Let $\sigma: \RR \to \RR$ and fix a dimension $n > 1$. A two-layered network with activation $\sigma$ is a function $N:\RR^n \to \RR$, of the form
	$$N(x) = \sum\limits_{i=1}^kc_i\sigma\left(u_i\cdot x\right),$$
	where $u_i \in \RR^n, c_i \in \RR$, for every $i = 1,\dots,k$. We will refer to $k$ as the width of the network. In most training procedures, it is typical to initialize the weight as \emph{i.i.d.} random vectors. Specifically, let $\{w_i\}_{i=1}^k$ be \emph{i.i.d.} as standard Gaussians in $\RR^n$ and let $\{s_i\}_{i=1}^k$ be \emph{i.i.d.} with $\PP(s_1 = 1) = \PP(s_1 = -1) = \frac{1}{2}$. We consider the random network,
	$$\mathcal{P}_k\sigma(x) := \frac{1}{\sqrt{k}} \sum\limits_{i=1}^ks_i\sigma\left(w_i\cdot x\right).$$
	Let $\Sph$ stand for the unit sphere in $\RR^n$. By restricting our attention to $x \in \RR^n$, with $\|x\| = 1$, we may consider $\mathcal{P}_k\sigma$ as a random process, indexed by $\Sph$. In other words, $\mathcal{P}_k\sigma$ is a random vector in $L^2(\Sph)$, equipped with its canonical rotation-invariant probability measure.
	
	A Gaussian process is a random vector $\mathcal{G} \in L^2(\Sph)$, such that for any finite set $\{x_j\}_{j=1}^m \subset \Sph$ the random vector $\{\mathcal{G}(x_j)\}_{j=1}^m \in \RR^m$, has a multivariate Gaussian law. Since $\mathcal{P}_k\sigma$ is a sum of independent centered vectors, standard reasoning suggests that as $k \to \infty$, $\mathcal{P}_k\sigma$ should approach a Gaussian law in $L^2(\Sph)$, i.e. a Gaussian process. Indeed, this is precisely Neal's CLT, which proves the existence of a Gaussian process $\mathcal{G}$, such that $\mathcal{P}_k\sigma \xrightarrow{k\to\infty} \mathcal{G}$, where the convergence is in distribution. 
	
	To make this result quantitative, we must first specify a metric. Our choice is inspired by the classical Wasserstein transportation in Euclidean spaces. The observant reader may notice that our definition, described below, does not correspond to the $p$-Wasserstein distance on $L^2(\mathbb{S}^{n-1})$, but rather the $p$-Wasserstein distance on $L^p(\mathbb{S}^{n-1})$. We chose this presentation for ease of exposition and its familiarity. 
	
	For $\mathcal{P},\mathcal{P}',$  random processes on the sphere, and $p \geq 1$ we define the functional $p$-Wasserstein distance as,
	$$ \mathcal{WF}_p(\mathcal{P},\mathcal{P}') :=\inf\limits_{(\mathcal{P},\mathcal{P}')} \left(\int\limits_{\Sph}\EE\left[\left|\mathcal{P}(x)-\mathcal{P}'(x)\right|^p\right]dx\right)^{\frac{1}{p}},$$
	where the infimum is taken over all couplings of $(\mathcal{P},\mathcal{P}')$ and where $dx$ is to be understood as the normalized uniform measure on $\Sph$. For $p = \infty$ we define $\mathcal{WF}_\infty$ as, 
	$$ \mathcal{WF}_\infty(\mathcal{P},\mathcal{P}') :=\inf\limits_{(\mathcal{P},\mathcal{P}')} \EE\left[\sup\limits_{x \in \Sph}\left|\mathcal{P}(x)-\mathcal{P}'(x)\right|\right].$$
	The notation $\mathcal{W}_p$ is reserved to the $p$-Wasserstein distance in finite-dimensional Euclidean spaces. Explicitly, if $X$ and $Y$ are random vectors in $\RR^N$, then
	$$\mathcal{W}_p(X,Y) :=\inf\limits_{(X,Y)} \EE\left[\|X-Y\|_2^p\right]^{\frac{1}{p}}.$$
	There is a straightforward way to connect between the quadratic functional Wasserstein distance on the sphere and the $L^2$ distance in Gaussian space.
	\begin{lemma} \label{lem:func_to_proc}
		Let $f,g: \RR \to \RR$, and $\gamma$, the standard Gaussian in $\RR$. Then,
		$$\mathcal{WF}^2_2(\mathcal{P}_kf,\mathcal{P}_kg)\leq \int\limits_{\RR}(f(x)-g(x))^2d\gamma(x).$$
	\end{lemma}
	\begin{proof}
		There is a natural coupling such that
		\begin{align*}
		\mathcal{WF}^2_2(\mathcal{P}_kf,\mathcal{P}_kg) &\leq \frac{1}{k}\int\limits_{\Sph}\EE\left[\left(\sum\limits_{i=1}^ks_i(f(w_i\cdot x)-g(w_i \cdot x))\right)^2\right]dx\\
		&=\frac{1}{k}\int\limits_{\Sph}\sum\limits_{i=1}^k\EE\left[\left(f(w_i\cdot x)-g(w_i \cdot x)\right)^2\right]dx\\
		&= \int\limits_{\RR}(f(x)-g(x))^2d\gamma(x).
		\end{align*}
		The first equality is a result of independence, while the second equality follows from the fact that for any $x\in \Sph$, $w_i\cdot x \sim \gamma$.
	\end{proof}
	\section{Results} \label{sec:results}
	We now turn to describe the quantitative CLT convergence rates obtained by our method.
	Our first result deals with polynomial activations. 
	\begin{theorem} \label{thm:polyactivs}
		Let $p(x) = \sum\limits_{m=0}^d a_mx^m$ be a degree $d$ polynomial. Then, 
		there exists a Gaussian process $\mathcal{G}$ on $\Sph$, such that
		$$\mathcal{WF}^2_\infty(\mathcal{P}_kp, \mathcal{G}) \leq C_d\max_{m}\{|a_m|^2\}\left(\frac{n^{5d - \frac{1}{2}}}{k}\right)^{\frac{1}{3}},$$
		where $C_d \leq d^{Cd}$, for some numerical constant $C > 0$.
	\end{theorem}
	According to the result, when the degree $d$ is fixed, as long as $k \gg n^{5d -\frac{1}{2}}$, $\mathcal{P}_kp$ is close to a Gaussian process. One way to interpret the metric $\mathcal{WF}_\infty$, in the result, is as follows. For any finite set $\{x_j\}_{j=1}^m \subset \left(\Sph\right)^m$, the random vector $\{\mathcal{P}_kp(x_j)\}_{j=1}^m \subset \RR^m$ converges to a Gaussian random vector, uniformly in $m$. Let us also mention that while the result is stated for Gaussian weights, the Gaussian plays no special role here (as will become evident from the proof), and the weights could be initialized by any symmetric random vector with sub-Gaussian tails. 
	
	One drawback of using polynomial activations is that the resulting network will always be a polynomial of bounded degree, which limits its expressive power. For this reason, in practice, neural networks are usually implemented using non-polynomial activations. By using a polynomial approximation scheme in Gaussian space, we are able to extend our result to this setting as well. We defer the necessary definitions and formulation of the result to Section \ref{sec:generalactivs}, but mention here two specialized cases of common activations.
	
	We first consider the Rectified Linear Unit (ReLU) function, denoted as $\psi(x) := \max(0,x)$. For this activation, we prove:
	\begin{theorem} \label{thm:reluactiv}
		There exists a Gaussian process $\mathcal{G}$ on $\Sph$, such that,
		$$\mathcal{WF}^2_2(\mathcal{P}_k\psi, \mathcal{G}) \leq C\left(\frac{\log(n)\log(\log(k))}{\log(k)}\right)^2,$$
		where $C > 0$ is a numerical constant.
	\end{theorem}
	The reader might get the impression that this is a weaker result than Theorem \ref{thm:polyactivs}. Indeed, the rate of convergence here is much slower. In order to get $\mathcal{WF}_2(\mathcal{P}_k\psi, \mathcal{G}) \leq \eps$, the theorem requires that $k \gtrsim n^{\frac{1}{\eps}\log\log n}$. Also, $\mathcal{WF}_2$ is a weaker metric than $\mathcal{WF}_\infty$. Let us point out that it may not be reasonable to expect similar behavior for polynomial and non-polynomial activations. The celebrated universal approximation theorem of Cybenko (\cite{cybenko1989approximation}, see also \cite{leshno1993multilayer,barron1993}) states that any function in $L^2(\Sph)$ can be approximated, to any precision, by a sufficiently wide neural network with a non-polynomial activation. Thus, as $k\to\infty$, the limiting support of $\mathcal{P}_k\psi$ will encompass all of $L^2(\Sph)$. This is in sharp contrast to a polynomial activation function, for which the support of $\mathcal{P}_kp$ is always contained in some finite-dimensional subspace of $L^2(\Sph)$, uniformly in $k$.
	
	Another explanation for the slow rate of convergence, is the fact that $\psi$ is non-differentiable. For smooth functions, the rate can be improved, but will still be typically sub-polynomial. As an example, we consider the hyperbolic tangent activation, $\tanh(x) = \frac{e^{x} - e^{-x}}{e^{x} + e^{-x}}$. 
	\begin{theorem} \label{thm:tanhactiv}
		There exists a Gaussian process $\mathcal{G}$ on $\Sph$, such that,
		$$\mathcal{W}^2_2(\mathcal{P}_k\tanh, \mathcal{G}) \leq C\exp\left(-\frac{1}{C}\sqrt{\frac{\log(k)}{\log(n)\log(\log(k))}}\right),$$
		where $c > 0$ is an absolute constant.
	\end{theorem}
	Finally, let us remark about possible improvements to our obtained rates. We do not know whether the constant $C_d$ in Theorem \ref{thm:monomctivs} is necessarily exponential, and we have made no effort to optimize it. We do conjecture that the dependence on the ratio $\frac{n^{5d - \frac{1}{2}}}{k}$ is not tight. To support this claim we prove an improved rate when the activation is monomial.

\begin{theorem} \label{thm:monomctivs}
	Let $p(x) = x^d$ for some $d \in \NN$. Then, 
	there exists a Gaussian process $\mathcal{G}$ on $\Sph$, such that
	$$\mathcal{WF}^2_{\infty}(\mathcal{P}_kp,\mathcal{G}) \leq C_d\frac{n^{2.5d-1.5}}{k},$$
	where $C_d \leq d^{Cd}$, for some numerical constant $C > 0$.
\end{theorem}
\begin{remark} \label{rmk:quadratic}
	It is plausible the dependence on $d$ and $k$ could be further improved. Let us note that when $d=2$, the best rate one could hope for is proportional to $\frac{n^{3}}{k}$. This is a consequence of the bounds proven in \cite{jiang15approx, bubeck16testing}, which show that if $n^3 \gg k$, then when considered as a random bi-linear form (or a Wishart matrix) $\mathcal{P}_kp$ is far from any Gaussian law. In fact, our proof of Theorem \ref{thm:monomctivs} can actually be improved when $d=2$ (or, in general, for even $d$), and we are able to obtain the sharp rate $\frac{n^3}{k}$.
	It is an interesting question to understand the correct rates when $d > 2$. 
\end{remark}
	\section{Polynomial processes} \label{sec:polynomialprocs} 
	For this section, fix a polynomial $p: \RR \to \RR$ of degree $d$, $p(x) = \sum\limits_{m=0}^d a_mx^m$. The goal of this section is to show that when $k$ is large enough, $\mathcal{P}_kp$ can be well approximated by a Gaussian process in the $\mathcal{WF}_\infty$ metric. Towards this, we will use the polynomial $p$ to embed $\RR^n$ into some high-dimensional tensor space. 
	\subsection{The embedding} 
	For $m\in \NN,$ we make the identification $(\RR^n)^{\otimes m} = \RR^{n^m}$ and focus on the subspace of symmetric tensors, which we denote $\mathrm{Sym}\left(\left(\RR^n\right)^{\otimes m}\right)$. If $\{e_i\}_{i=1}^n$ is the standard orthonormal basis of $\RR^n$, then an orthonormal basis for $\mathrm{Sym}\left(\left(\RR^n\right)^{\otimes m}\right)$, is given by the set
	$$\{e_I|I \in \mathrm{MI}_n(m)\}.$$
	where $\mathrm{MI}_n(m)$ is the set of multi-indices,
	$$\mathrm{MI}_n(m) = \{(I_1,\dots I_n) \in \NN^n | I_1+\dots +I_n= m\},$$
	With this perspective, we have $e_I = \otimes_{i=1}^n\left(e_{i}^{\otimes I_i}\right),$ and we denote the inner product on $\mathrm{Sym}\left(\left(\RR^n\right)^{\otimes m}\right)$ by $\langle \cdot, \cdot \rangle_m$. We also use the following multi-index notation: if $x =(x_1,\dots, x_n)\in \RR^n$, we denote $x^I = \prod\limits_{i=1}^nx_i^{I_i}$.\\
	Define the feature space $H := \oplus_{m=0}^d \mathrm{Sym}\left(\left(\RR^n\right)^{\otimes m}\right)$. 
	 If $\pi_m: H\to \mathrm{Sym}\left(\left(\RR^n\right)^{\otimes m}\right)$ is the natural projection, then an inner product on $H$ may be defined by 
	$$\langle v, u\rangle_H := \sum\limits_{m=0}^d \langle \pi_mv, \pi_mu\rangle_m.$$
	We further define the embedding $P:\RR^n \to H$, $P(x) = \sum\limits_{m=0}^d \sqrt{|a_m|}x^{\otimes m}$, which induces a bi-linear form on $H$ as,
	$$Q(u,v) := \sum\limits_{m=0}^d\mathrm{sign}(a_m)\langle \pi_m u, \pi_m v\rangle_m.$$
	Observe that $Q$ is not necessarily positive definite, but still satisfies the following Cauchy-Schwartz type inequality,
	\begin{align} \label{eq:QCS}
	Q(u,v) \leq \|u\|_H\|v\|_H.
	\end{align}
	Furthermore, it is clear that for any $x, y \in \RR^n$,
	$$Q(P(x),P(y)) = \sum\limits_{m=0}^d a_m (x \cdot y)^m = p\left( x\cdot y\right),$$
	and we have the identity,
	\begin{equation} \label{eq:feature_embedding}
	\mathcal{P}_kp(x)=\frac{1}{\sqrt{k}}\sum\limits_{i=1}^ks_ip(w_i\cdot x) = \frac{1}{\sqrt{k}}\sum\limits_{i=1}^ks_iQ(P(x) ,P(w_i))= Q\left(P(x),\frac{1}{\sqrt{k}}\sum\limits_{i=1}^ks_iP(w_i)\right).
	\end{equation}
	Consider the random vector $X_k:=\frac{1}{\sqrt{k}}\sum\limits_{i=1}^ks_iP(w_i)$ taking values in $H$. By the central limit theorem, we should expect $X_k$ to approach a Gaussian law. The next result shows that approximate Gaussianity of $X_k$ implies that the process $\mathcal{P}_kp$ is approximately Gaussian as well.
	\begin{lemma} \label{clm:embedding_clt}
		Let $G$ be a Gaussian random vector in $H$ and define the random process on $\mathcal{G}$ in $\Sph$ by $\mathcal{G}(x) := Q(P(x),G)$. Then, $\mathcal{G}$ is a Gaussian process and,
		$$\mathcal{WF}^2_\infty(\mathcal{P}_kp,\mathcal{G}) \leq \left(\sum\limits_{m=0}^d|a_m|\right) \mathcal{W}^2_2(X_k,G).$$
	\end{lemma}
	\begin{proof}
		Let $(X_k,G)$ be the optimal coupling so that $\mathcal{W}_2^2(X_k, G) = \EE\left[\|X_k - G\|_H^2\right].$
		We then have
		\begin{align*}
		\mathcal{WF}_\infty(\mathcal{P}_kp, \mathcal{G}) &\leq \EE\left[\sup\limits_{x \in \Sph}\left|\mathcal{P}_kp(x) - \mathcal{G}{(x)} \right|\right] = \EE\left[\sup\limits_{x \in \Sph}\left|Q(P(x), X_k - G)\right|\right]\\
		&\leq \sup\limits_{x \in \Sph}\|P(x)\|_H\sqrt{\EE\left[\|X_k - G\|_H^2\right]} = \sup\limits_{x \in \Sph}\|P(x)\|_H\cdot \mathcal{W}_2(X_k, G),
		\end{align*}
		where we have used \eqref{eq:QCS} in the second inequality.
		Now, for any $x \in \Sph$,
		$$\|P(x)\|_H = \sqrt{\sum\limits_{m=0}^d |a_m|\langle x^{\otimes m}, x^{\otimes m} \rangle_m} =\sqrt{\sum\limits_{m=0}^d|a_m|}.$$
	\end{proof}

	So, we wish to show that the random vector $X_k:=\frac{1}{\sqrt{k}}\sum\limits_{i=\mathrm{I}}^ks_iP(w_i)$ is approximately Gaussian inside $H$. 
	For this, we will apply the following Wasserstein CLT bound, recently proven by Bonis, in \cite{bonis2020}.
	\begin{theorem}{\cite[Theorem 1]{bonis2020}} \label{thm:bonis}
		Let ${Y}_i$ be i.i.d isotropic random vectors in $\RR^N$ and let $G$ be the standard Gaussian. Then, if $S_k = \frac{1}{\sqrt{k}}\sum\limits_{i=1}^kY_i$,
		$$\mathcal{W}_2^2(S_k, G) \leq \frac{\sqrt{N}}{k}\norm{\EE\left[YY^T\norm{Y}_2^2\right]}_{HS}.$$
	\end{theorem}
	Since the theorem applies to isotropic random vectors, for which the covariance matrix is the identity, we first need to understand $\Sigma :=\mathrm{Cov}\left(P(w)\right)$. Let us emphasize the fact that $\Sigma$ is a bi-linear operator on $H$. Thus it can be regarded as a $\mathrm{dim}(H) \times \mathrm{dim}(H)$ positive semi-definite matrix.
	\subsection{The matrix $\Sigma$}
	We first show that one may disregard small eigenvalues of $\Sigma$. Let $(\lambda_j,v_j)$ stand for the eigenvalue/vector pairs of $\Sigma$.	Fix $\delta >0$ define $V_\delta = \mathrm{span}\left(v_j|\lambda_j \leq \delta\right)$ and let $\Pi_\delta,\Pi_\delta^\perp$ be the orthogonal projection unto $V_\delta, V_\delta^\perp$, respectively.
	\begin{lemma} \label{lem:projection}
		Let $G \sim \mathcal{N}(0, \Sigma)$ be a Gaussian in $H$, then
		$$\mathcal{W}_2^2(X_k,G)\leq \mathcal{W}_2^2(\Pi_\delta^\perp X_k,\Pi_\delta^\perp G)+8n^d\delta.$$
	\end{lemma}
	\begin{proof}
		For any coupling $(X_k,G)$ we have 
		\begin{align*}
		\mathcal{W}_2^2(X_k,G)&\leq \EE\left[\norm{X_k - G}^2\right] 
		= \EE\left[\norm{\Pi_\delta X_k - \Pi_\delta G}^2\right] + \EE\left[\norm{\Pi_\delta^\perp X_k - \Pi_\delta^\perp G}^2\right]\\
		&\leq 2\EE\left[\norm{\Pi_\delta G}^2\right] + 2\EE\left[\norm{\Pi_\delta X_k}^2\right] + \EE\left[\norm{\Pi_\delta^\perp X_k - \Pi_\delta^\perp G}^2\right]\\
		&\leq 4\dim(H)\delta + \EE\left[\norm{\Pi_\delta^\perp X_k - \Pi_\delta^\perp G}^2\right].
		\end{align*}
		The proof concludes by taking the coupling for which $\Pi_\delta^\perp X_k,\Pi_\delta^\perp G$ is optimal, and by noting $\dim(H)\leq 2n^d$.
	\end{proof}
	Next, we bound from above the eigenvalues of $\Sigma$.
\begin{lemma} \label{lem:sigma_bound}
	Let $\Sigma = \mathrm{Cov}(P(w))$, where $P(w)$ is defined as in $\eqref{eq:feature_embedding}$. Then
	$$\norm{\Sigma}_{op} \leq (4d)!\max\limits_m\{|a_m|\}n^{\frac{d-1}{2}}  .$$
\end{lemma}
\begin{proof}
	Let $\sum\limits_I v_Ie_I = v \in H$ be a unit vector, we wish to bound $\inner{v}{\Sigma v} = \mathrm{Var}\left(\langle P(w),v\rangle_H\right)$ from above. Let us denote the degree $d$ polynomial $\sum\limits_{i=0}^m\sum\limits_{I \in \mathrm{MI}_n(m)}\sqrt{|a_m|}v_Ix^I = q(x)= \langle P(x), v \rangle_H$. We will prove the claim by induction on $d$. The case $d = 1$, is rather straightforward to check. For the general case, we will use the Gaussian Poincar\'e inequality (see \cite[Proposition 1.3.7]{nourdin12normal}, for example) to reduce the degree. According to the inequality,
	\begin{align} \label{eq:varDecomp}
	\mathrm{Var}\left(\langle P(w),v\rangle_H\right) = \mathrm{Var}\left(q(w)\right) &\leq \EE\left[\|\nabla q(w)\|_2^2\right]\nonumber \\
	&= \sum\limits_{i=1}^n\EE\left[\left|\frac{d}{dx_i} q(w)\right|^2\right]\nonumber \\
	 &= \sum\limits_{i=1}^n\mathrm{Var}\left(\frac{d}{dx_i}q(w)\right)+\EE\left[\frac{d}{dx_i}q(w)\right]^2.
	\end{align}
	Fix $i =1,\dots, n$, if $I \in \mathrm{MI}_n(m)$ we denote by $\partial_iI \in \mathrm{MI}_n(m-1)$, to be a multi-index set such that
	$$\partial_iI_j = \begin{cases}
	I_j & \text{if } i \neq j\\
	\max(0, I_i - 1)&\text{if } i = j.
	\end{cases}$$
	With this notation, we have,
	\begin{align*}
	\frac{d}{dx_i}q(w) &= \frac{d}{dx_i}\left(\sum\limits_{m=0}^{d}\sum\limits_{I \in \mathrm{MI}_n(m)} \sqrt{|a_{m}|}w^Iv_I\right)\\
	&= \sum\limits_{m=0}^{d}\sum\limits_{I \in \mathrm{MI}_n(m)} \sqrt{|a_{m}|}I_iw^{\partial_iI}v_{I}.
	\end{align*}
	Since $\frac{d}{dx_i}q$ is a polynomial of degree $d-1$, we thus get by induction,
	$$\mathrm{Var}\left(\frac{d}{dx_i}q(w)\right) \leq (4d-4)! \max\limits_m\{|a_m|\}n^\frac{d-2}{2}\sum\limits_{m=0}^{d}\sum\limits_{I \in \mathrm{MI}_n(m)} I_i^2v_{I}^2.$$
	Observe that $I_i \leq d$ and that for every $I \in \mathrm{MI}_n(m)$, there are at most $d$ different indices $i \in [n]$, for which $I_i \neq 0$. Therefore,
	\begin{equation} \label{eq:poincareBound}
	\sum\limits_{i=1}^n\mathrm{Var}\left(\frac{d}{dx_i}q(w)\right) \leq d^2(4d-4)!\max\limits_m\{|a_m|\}n^{\frac{d-2}{2}}\left(d\sum\limits_{I}v_I^2\right) \leq (4d-1)!\max\limits_m\{|a_m|\}n^{\frac{d-2}{2}}.
	\end{equation}
	Furthermore, if for some $j \in [n]$, $\partial_iI_j$ is odd, then $\EE\left[w^{\partial_iI}\right] = 0$. Otherwise,  $$|\EE\left[w^{\partial_iI}\right]| \leq |\EE\left[w_1^{d-1}\right]|\leq \sqrt{d!}.$$
	It is easy to verify that the size of the following set,
	$$A_i = \{I \in \cup_{m=0}^d\mathrm{MI}_n(m)|I_i\EE[w^{\partial_iI}]\neq 0\},$$
	is at most $(2n)^{\frac{d-1}{2}}.$
	Thus, since there are at most $(2n)^{\frac{d-1}{2}}$ elements which do not vanish, Cauchy-Schwartz's inequality shows,
\begin{align*} 
	\EE\left[\frac{d}{dx_i}q(w)\right]^2 &\leq d^2\max\limits_m\{|a_m|\}\EE\left[\sum\limits_{m=0}^{d} \sum\limits_{I \in \mathrm{MI}_n(m)} w^{\partial_iI}v_{I}\right]^2 \nonumber\\
	&\leq  (4d-1)!\max\limits_m\{|a_m|\}n^{\frac{d-1}{2}}\sum\limits_{I \in A_i}v^2_{I}. \nonumber\\
\end{align*}
Note that if $I \in A_i$, then necessarily $I_i$ is odd. In this case, it follows that for $j \neq i$, $\EE\left[w^{\partial_j I}\right] = 0$. Hence, $A_i \cap A_j = \emptyset,$ and
\begin{align} \label{eq:gradientBound}
\sum\limits_{i=1}^n\EE\left[\frac{d}{dx_i}q(w)\right]^2 &\leq(4d-1)!\max\limits_m\{|a_m|\} n^\frac{d-1}{2}\sum\limits_{i=1}^n\sum\limits_{I \in A_i}v^2_{I}\nonumber\\
&\leq (4d-1)!\max\limits_m\{|a_m|\}n^{\frac{d-1}{2}}\sum\limits_{I}v_I^2 =(4d-1)!\max\limits_m\{|a_m|\}n^{\frac{d-1}{2}}.
\end{align}
We now plug \eqref{eq:poincareBound} and \eqref{eq:gradientBound} into \eqref{eq:varDecomp} to obtain
$$\mathrm{Var}\left(\langle P(w),v\rangle_H\right) \leq (4d)!\max\limits_m\{|a_m|\}n^{\frac{d-1}{2}}.$$
\end{proof}
	Remark that, up to the multiplicative dependence on $d$, this bound is generally sharp. As an example, when $d=2\ell - 1$ is odd, one can consider the degree $d$ polynomial, $$q(x) = \frac{1}{n^{\ell/2}}\sum\limits_{i_1,\dots i_\ell=1}^n x_{i_1}x_{i_2}^2\dots x_{i_\ell}^2.$$ For this polynomial it may be verified that 
	$\mathrm{Var}(q(w)) = \Omega(n^{\ell-1}) = \Omega\left(n^{\frac{d-1}{2}}\right).$

	\subsection{A functional CLT for polynomial processes}

\begin{proof}[Proof of Theorem \ref{thm:polyactivs}]
	Let $\delta$ be some small number to be determined later and set $\tilde{X}_k = \Sigma^{-1/2}X_k$ and $\tilde{G}$, the standard Gaussian in $H$. By Lemma \ref{lem:projection},
	\begin{align*}
	\mathcal{W}_2^2(X_k, G) &\leq \mathcal{W}_2^2(\Pi_\delta^\perp X_k, \Pi_\delta^\perp G)  + 8n^d\delta\\
	 &= \mathcal{W}_2^2(\Sigma^{1/2}\Pi_\delta^\perp \tilde{X}_k, \Sigma^{1/2}\Pi_\delta^\perp \tilde{G})  + 8n^d\delta \leq \|\Sigma\|_{op}\mathcal{W}_2^2(\Pi_\delta^\perp \tilde{X}_k, \Pi_\delta^\perp \tilde{G})  + 8n^d\delta.
	\end{align*}
	We focus on the term $\mathcal{W}_2^2(\Pi_\delta^\perp\tilde{X}_k, \Pi_\delta^\perp \tilde{G})$ for which Theorem \ref{thm:bonis} may be invoked,
	\begin{align*}
\mathcal{W}_2^2(\Pi_\delta^\perp\tilde{X}_k, \Pi_\delta^\perp \tilde{G}) &\leq 
	 \frac{\sqrt{\mathrm{dim}(H)}}{k}\EE\left[\norm{\Pi_\delta^\perp \Sigma^{-1/2}P(w)}_H^4\right]\\
	&\leq  \frac{\sqrt{\mathrm{dim}(H)}}{k}\EE\left[\norm{P(w)}_H^4\right]\norm{\Pi_\delta^\perp\Sigma^{-1}}_{op}^2\\
	&\leq \frac{\sqrt{\mathrm{dim}(H)}}{\delta^2k} \EE\left[\norm{P(w)}_H^4\right].
	\end{align*}
	In the first inequality, we have used Jensen's inequality on the bound from Theorem \ref{thm:bonis}.
	Let us estimate $\EE\left[\norm{P(w)}_H^4\right]$. By definition,
	\begin{align*}
	\EE\left[\norm{P(w)}_H^4\right] &= \EE\left[\left(\sum\limits_{m=0}^d|a_m|\|w^{\otimes m}\|^{2m}_m\right)^2\right] \leq  \left(\sum\limits_{m=0}^da_m^2\right)\left(\sum\limits_{m=0}^d\EE\left[\|w\|^{4m}_2\right]\right)\\
	 &\leq \left(\sum\limits_{m=0}^da_m^2\right)\left(\sum\limits_{m=0}^d(2m)!\left(4\EE\left[\|w\|^2_2\right]\right)^{2m}\right) \leq \left(\sum\limits_{m=0}^da_m^2\right)\left(\sum\limits_{m=0}^d(2m)!\left(4n\right)^{2m}\right)\\
	 &\leq \left(\sum\limits_{m=0}^da_m^2\right) 16^d(2d)!n^{2d} \leq \max_{m} \{a_m^2\} (100d)!n^{2d}
	\end{align*}
	The first inequality is Cauchy-Schwartz and in the second inequality we have used the fact that $\|w\|_2$ has sub-exponential tails.
	
	Since $\mathrm{dim}(H) \leq 2n^d$, it follows that,
	$$\mathcal{W}_2^2(X_k, G) \leq \|\Sigma\|_{op}\frac{ (100d)!n^{\frac{5d}{2}}}{\delta^2k}\max_{m} \{a_m^2\} + 8n^d\delta.$$
	We plug the estimate for $\|\Sigma\|_{op}$ from Lemma \ref{lem:sigma_bound} to deduce:
	$$\mathcal{W}_2^2(X_k, G) \leq \frac{ (110d)!n^{3d-0.5}}{\delta^2k}\max_{m} \{|a_m|^3\} + 8n^d\delta.$$
	We now take $\delta = \left(\frac{(110d)!n^{2d-0.5}\max_{m}\{|a_m|^3\}}{k}\right)^{\frac{1}{3}}$ to obtain 
	$$\mathcal{W}_2^2(X_k, G) \leq 16 \max_{m}\{|a_m|\}\left((110d)!\right)\frac{n^{\frac{5d - 0.5}{3}}}{k^{\frac{1}{3}}}$$
	To finish the proof, define the Gaussian process $\mathcal{G}$ by $\mathcal{G}(x) = Q\left(P(x),G\right)$, and invoke Claim \ref{clm:embedding_clt}.
	\end{proof}
	\subsection{An improved rate for tensor powers}
	Throughout this section we assume that $p(x) = x^d$ for some $d \in \NN$. Under this assumption, we improve Theorem \ref{thm:polyactivs}. This improvement is enabled by two factors:
	\begin{itemize}
		\item A specialized CLT for tensor powers, as proven in \cite{mikulincer2020clt}.
		\item An improved control on the eigenvalues of $\Sigma$, which allows to bypass Lemma \ref{lem:projection}. 
	\end{itemize}
	Let us first state the result about approximating tensor powers by Gaussians. Note that for a polynomial $p$ as above, we have the embedding map $P(x) = x^{\otimes d}$. Since the image of $P$ is always a symmetric $d$-tensor, we allow ourselves to  restrict the embedding map $P$ and overload notations, so that $P:\RR^n \to \mathrm{Sym}\left(\left(\RR^n\right)^{\otimes d}\right)$. In this case, for $w \sim \mathcal{N}\left(0, \mathrm{I}_d\right)$, we have $\Sigma := \mathrm{Cov}(P(w))$, and $X_k:= \frac{1}{\sqrt{k}}\sum\limits_{i=1}^ks_iP(w_i)$.
	\begin{theorem}{\cite[Theorems 2 and 5]{mikulincer2020clt}} \label{thm:tensorPowerCLT}
		Let the above notations prevail. Then, there exists a Gaussian random vector $G$, in $\mathrm{Sym}\left(\left(\RR^n\right)^{\otimes d}\right)$, such that,
		$$\mathcal{W}_2^2\left(X_k, G\right) \leq C_d\|\Sigma\|_{op}\|\Sigma^{-1}\|_{op}^2\frac{n^{2d-1}}{k},$$
		where $C_d = d^{Cd}$, for some universal constant $C > 0$.
	\end{theorem}
Remark that the result in \cite{mikulincer2020clt} actually deals with the random vector $\sqrt{\Sigma^{-1}}X_k$. Since we care about the un-normalized vector $X_k$ we incur a dependence on $\|\Sigma\|_{op}$. We now show how to bound from below the eigenvalues of $\Sigma$.
\begin{lemma} \label{lem:lowerSigmaPower}
	Let $\lambda_\mathrm{min}(\Sigma)$ stand for the minimal eigenvalue of $\Sigma$. Then
	$$\lambda_\mathrm{min}(\Sigma) \geq \frac{1}{d!}.$$
\end{lemma}
\begin{proof}
	Let $v \in \mathrm{Sym}\left(\left(\RR^n\right)^{\otimes d}\right)$ be a unit vector. We can thus write $v = \sum\limits_{|I| = d} v_Ie_I$, with $\sum\limits_{I \in \mathrm{MI}_n(d)} v_I^2 = 1$. Define the degree $d$ homogeneous polynomial $q:\RR^n \to \RR$, by $q(x) = \sum\limits_{I \in \mathrm{MI}_n(d)}v_Ix^I$.  In this case we have $\langle v, P(w)\rangle = q(w)$, and it will be enough to show, 
	$$\mathrm{Var}\left(\langle v, P(w)\rangle\right) = \mathrm{Var}\left(q(w)\right) \geq \frac{1}{d!}.$$
	We will use the variance expansion for functions of Gaussian vectors, which can be found at \cite[Proposition 1.5.1]{nourdin12normal}. According to this expansion, for any smooth enough function $f:\RR^n \to \RR$,
	\begin{equation} \label{eq:nourdinDecomp}
	\mathrm{Var}(f(w)) = \sum\limits_{m=1}^{\infty}\frac{\left\|\EE\left[\nabla^mf(w)\right]\right\|_m^2}{m!}.
	\end{equation}
	Here $\nabla^mf$ is the $m^{th}$ total derivative of $f$, which we regard as an element in $\left(\RR^n\right)^{\otimes m}$. In particular, we have,
	$$\mathrm{Var}\left(q(w)\right) \geq \frac{\left\|\EE\left[\nabla^dq(w)\right]\right\|_d^2}{d!}.$$
	Now, if $I \neq J$ are two multi-subsets of $[n]$, with $I,J \in \mathrm{MI}_n(d)$, we have 
	$$\frac{d}{dx^I}x^J = 0 \text{ and } \frac{d}{dx^I}x^I = I!.$$
	So, since $\nabla^df = \{\frac{d}{dx^I}q\}_{I \in \mathrm{MI}_n(d)}$,
	$$\|\EE\left[\nabla^df(w)\right]\|_d^2 = \sum\limits_{I \in \mathrm{MI}_n(d)}I!v_{I}^2 \geq 1,$$
	and
	$$\mathrm{Var}\left(q(w)\right) \geq \frac{1}{d!},$$\
	as required.
\end{proof}
We are now in a position to prove Theorem \ref{thm:monomctivs}.
\begin{proof}[Proof of Theorem \ref{thm:monomctivs}]
	By combining Lemma \ref{lem:sigma_bound} and Lemma \ref{lem:lowerSigmaPower}, there exists some numerical constant $C' >0$, such that $$\|\Sigma\|_{op}\|\Sigma^{-1}\|^2_{op} \leq d^{C'd}n^\frac{d-1}{2}.$$ 
	Thus, Theorem \ref{thm:tensorPowerCLT} shows  that there exists a Gaussian vector $G$ in $\mathrm{Sym}\left(\left(\RR^n\right)^{\otimes d}\right)$, such that
	$$\mathcal{W}_2^2\left(X_k, G\right) \leq d^{Cd}\frac{n^{2.5d-1.5}}{k},$$
	for some other constant $C > 0$.
	Define the Gaussian process $\mathcal{G}(x) = \langle P(x), G\rangle$, then Lemma \ref{lem:func_to_proc} shows,
	$$\mathcal{WF}^2_{\infty}(\mathcal{P}_kp,\mathcal{G}) \leq d^{Cd}\frac{n^{2.5d-1.5}}{k},$$
	which concludes the proof.
	When $d=2$, it is not hard to see that $\norm{\Sigma}_{op}$ can be bounded by an absolute constant (see Lemma \ref{lem:improvedQuadratic} in the appendix). In this case, 
	$$\|\Sigma\|_{op}\|\Sigma^{-1}\|^2_{op} \leq C,$$
	which is the reason behind Remark \ref{rmk:quadratic}.
\end{proof}
\section{General activations} \label{sec:generalactivs}
In this section we consider a general (non-polynomial) activation function $\sigma: \RR \to \RR$. Our goal is to derive a quantitative CLT for the random process $\mathcal{P}_k\sigma$. Our strategy will be to approximate $\sigma$ by some polynomial, for which Theorem \ref{thm:polyactivs} applies. We set $\gamma$ to be the law of the standard Gaussian in $\RR$. Lemma \ref{lem:func_to_proc} suggests that, in order to control the remainder in the approximation,  it would be beneficial to find a polynomial $p$, such that $p$ and $\sigma$ are close in $L^2(\gamma)$.

In $L^2(\gamma)$ there is a distinguished set of polynomials, the so-called Hermite polynomials. Henceforth we denote $h_m$ to be the $m^{th}$ normalized Hermite polynomial,
$$h_m(x) = \frac{(-1)^m}{\sqrt{m!}}\left(\frac{d^m}{dx^m}e^{-\frac{x^2}{2}}\right)e^{\frac{x^2}{2}}.$$

The reader is referred to \cite{janson1997gaussian} for the necessary definitions and proofs pertaining to Hermite polynomials. We will mainly care about the fact that $\{h_m\}_{m=0}^{\infty}$ forms a complete orthonormal system in $L^2(\gamma)$. Thus, assuming that $\sigma \in L^2(\gamma)$, it may be written as,
$$\sigma = \sum\limits_{m=0}^\infty \hat{\sigma}_mh_m\text{ , where } \hat{\sigma}_m := \int\limits_{\RR}\sigma(x)h_m(x)d\gamma(x).$$ 
Let us also define the remainder function of $\sigma$ as,
$$R_\sigma(d) = \sum\limits_{m=d+1}^\infty \hat{\sigma}_r^2.$$
If we define the degree $d$ polynomial 
\begin{equation} \label{eq:sigmapolynom}
p_d := \sum\limits_{m=1}^d \hat{\sigma}_mh_m,
\end{equation} 
we then have,
\begin{equation} \label{eq:sigmaapprox}
\|\sigma - p_d\|^2_{L^2(\gamma)} \leq R_\sigma(d).
\end{equation}
With these notations, the main result of this section is:
\begin{theorem} \label{thm:generalproc}
	Suppose that $\sigma \in L^2(\gamma)$. Then, there exists a Gaussian process $\mathcal{G}$ on $\Sph$, such that,
		$$\mathcal{WF}^2_2(\mathcal{P}_k\sigma, \mathcal{G}) \leq C'\frac{\max\limits_m |\hat{\sigma}_m|^2}{k^{\frac{1}{6}}} + R_\sigma\left(\frac{\log(k)}{C'\log(n)\log(\log(k))}\right),$$
		where $C' > 0$ is a numerical constant.
\end{theorem}
Before proving the theorem, we first focus on the coefficients of the polynomial $p_d$, defined in \eqref{eq:sigmapolynom}, with respect to the standard monomial basis. For this, we write $h_m$, explicitly (see \cite[Chapter 3]{janson1997gaussian}) as,
$$h_m(x) = \sqrt{m!} \sum\limits_{j=0}^\frac{m}{2}\frac{(-1)^j}{j!(m-2j)!2^j}x^{m-2j}.$$
Write now $p_d = \sum\limits_{m=0}^da_mx^m$ and let us estimate $a_m$. 
\begin{lemma} \label{lem:sigmacoeffs}
	It holds that
	$$|a_m| \leq \max_{i}|\hat{\sigma}_i|\frac{2}{\sqrt{m!}}2^d.$$
\end{lemma}
\begin{proof}
	We have:
	\begin{align*}
	|a_m| &\leq \sum\limits_{i=m}^d|\hat{\sigma}_i|\frac{\sqrt{i!}}{m!((i-m)/2)!2^{(i-m)/2}}\\
	&\leq \max\limits_{i}|\hat{\sigma}_i|\sum\limits_{i=m}^d\frac{\sqrt{i!}}{m!((i-m)/2)!2^{(i-m)/2}}\\
	&\leq \max\limits_{i}|\hat{\sigma}_i|\sum\limits_{i=m}^d\frac{\sqrt{i!}}{m!\sqrt{(i-m)!}}
	= \max\limits_{i}|\hat{\sigma}_i|\sum\limits_{i=m}^d\frac{1}{\sqrt{m!}}\sqrt{\binom{i}{m}} \\
	&\leq \max\limits_{i}|\hat{\sigma}_i|\frac{1}{\sqrt{m!}}\sum\limits_{i=m}^d\binom{i}{m}
	= \max\limits_{i}|\hat{\sigma}_i|\frac{1}{\sqrt{m!}} \binom{d+1}{m+1}\leq \max\limits_{i}|\hat{\sigma}_i|\frac{2}{\sqrt{m!}}2^d,
	\end{align*}
	where the last equality is Pascal's identity.
\end{proof}
We may now prove Theorem \ref{thm:generalproc}.
\begin{proof}[Proof of Theorem \ref{thm:generalproc}]
	Fix $d$ and let $\mathcal{G}$ be the Gaussian process promised by Theorem \ref{thm:polyactivs}, for $p_d$.
	By the triangle inequality, 
	$$\mathcal{WF}^2_2(\mathcal{P}_k\sigma, \mathcal{G}) \leq 2\mathcal{WF}^2_2(\mathcal{P}_k\sigma, \mathcal{P}_kp_d) + 2\mathcal{WF}^2_{2}(\mathcal{P}_kp_d, \mathcal{G}) \leq 2\mathcal{WF}^2_\infty(\mathcal{P}_k\sigma, \mathcal{P}_kp_d) + 2\mathcal{WF}^2_2(\mathcal{P}_kp_d, \mathcal{G}).$$
	We now invoke Lemma \ref{lem:func_to_proc} with \eqref{eq:sigmaapprox} to obtain, 
	\begin{align*}
	\mathcal{WF}^2_2(\mathcal{P}_k\sigma, \mathcal{P}_kp_d) &\leq \|p_d - \sigma\|^2_{L^2(\gamma)} \leq R_\sigma(d).\\
	\end{align*}
	For the other term, Theorem \ref{thm:polyactivs} along with Lemma \ref{lem:sigmacoeffs} imply,
	\begin{align*}
	\mathcal{WF}^2_\infty(\mathcal{P}_kp_d, \mathcal{G}) &\leq \max\limits_i |\hat{\sigma_i}|^2\cdot d^{Cd}\frac{n^{2d}}{k^{\frac{1}{3}}},
	\end{align*}
	for some numerical constant $C > 0$.
	So,
	$$\mathcal{WF}^2_2(\mathcal{P}_k\sigma, \mathcal{G}) \leq 2\max\limits_i |\hat{\sigma_i}|^2d^{Cd}\frac{n^{2d}}{k^{\frac{1}{3}}} + 2R_\sigma(d).$$
	Finally, choose $d = \lceil\frac{\log(k)}{100C\log(n)\log(\log(k))}\rceil$. It can be verified that for any $\delta > 0, \alpha>0$,
	$$d^{Cd}\cdot n^{2d}\leq\log(k)^{\frac{\log(k)}{100\log(\log(k))}}\cdot e^{\frac{\log(k)}{10}} = O(k^\frac{1}{6}).$$ This implies the existence of an absolute constant $C' > 0$,  for which,
	$$\mathcal{WF}^2_2(\mathcal{P}_k\sigma, \mathcal{G}) \leq C'\left(\frac{\max\limits_i |\hat{\sigma_i}|^2 }{k^{\frac{1}{6}}} + R_\sigma(d)\right).$$
	The proof is complete.
\end{proof}
\subsection{ReLU activation}
In this section we specialize Theorem \ref{thm:generalproc} to the ReLU activation  $\psi(x) := \max(0,x)$. 
The calculation of $\hat{\psi}_m$ may be found in \cite{daniely2016toward,goel2019time}. We repeat it here for completeness. 
\begin{lemma} \label{lem:relucoeff}
	Let $m \in \NN$. Then,
	\begin{equation} \label{eq:psihat}
	|\hat{\psi}_m| = \begin{cases}
	\frac{1}{\sqrt{2}} & m = 1\\
	0 & m>1 \text{ and odd}\\
	\frac{(m-3)!!}{\sqrt{\pi}\sqrt{m!}} & \text{otherwise}
	\end{cases}.
	\end{equation}
	In particular, $|\hat{\psi}_m| \leq \frac{1}{m^\frac{3}{2}},$
	and
	$$R_\psi(d) \leq \frac{1}{d^2}.$$
\end{lemma}
	\begin{proof}
		Note that once \eqref{eq:psihat} is established the rest of the proof is trivial. Thus, let us focus on calculating $\hat{\psi}_m$. We will use the following formula for the derivative of Hermite polynomials,
		\begin{equation} \label{eq:hermitederivative}
		h'_m(x) = \sqrt{m}h_{m-1}(x).
		\end{equation}
		Using this, we have, with an application of integration by parts,
		\begin{align*}
		\hat{\psi}_m = \int\limits_{\RR}h_m(x)\psi(x)d\gamma(x) &= \int\limits_{x > 0}h_m(x)xd\gamma(x) = \frac{h_m(0)}{\sqrt{2\pi}} - \int \limits_{x >0}h_m'(x)d\gamma(x)\\
		&= \frac{h_m(0)}{\sqrt{2\pi}} - \sqrt{m}\int\limits_{x >0}h_{m-1}(x)d\gamma(x)\\
		&= \frac{h_m(0)}{\sqrt{2\pi}} + (-1)^{m}\sqrt{\frac{m}{2\pi(m-1)!}}\int\limits_{x >0}\frac{d^{m-1}}{dx^{m-1}}e^{-\frac{x^2}{2}}(x)dx\\
		&= \frac{h_m(0)}{\sqrt{2\pi}} + (-1)^{m} \sqrt{\frac{m}{2\pi(m-1)!}}\frac{d^{m-2}}{dx^{m-2}}e^{-\frac{x^2}{2}}(0)\\
		&= \frac{h_m(0)+ \sqrt{\frac{m}{(m-1)}}h_{m-2}(0)}{\sqrt{2\pi}}.
		\end{align*}
		For $h_m(0)$, the following explicit formula holds:
		$$h_m(0) = \begin{cases}
		0 & \text{for } m \text{ odd}\\
		(-1)^{m/2}\frac{(m-1)!!}{\sqrt{m!}} & \text{for } m \text{ even}
		\end{cases}.$$
		In this case, for $m$ even,
		$$\sqrt{\frac{m}{(m-1)}}h_{m-2}(0) = (-1)^{m/2 - 1}\sqrt{\frac{m}{m-1}}\frac{(m-3)!!}{\sqrt{(m-2)!}}=(-1)^{m/2 - 1}\frac{m(m-3)!!}{\sqrt{(m-1)!}},$$
		and \eqref{eq:psihat} follows.
	\end{proof}

Theorem \ref{thm:reluactiv} follows immediately, by plugging the above Lemma into Theorem \ref{thm:generalproc}.
\begin{proof}[Proof of Theorem \ref{thm:reluactiv}]
	From Lemma \ref{lem:relucoeff} we see that $\max\limits_i|\hat{\psi}_i| \leq 1$, and so coupled with Theorem \ref{thm:generalproc}, we get
	$$\mathcal{W}^2_2(\mathcal{P}_k\sigma, \mathcal{G}) \leq C\left(\frac{1}{k^{\frac{1}{6}}} + \left(\frac{\log(n)\log(\log(k))}{\log(k)}\right)^2\right).$$
	It is now enough to observe,
	$$\frac{1}{k^\frac{1}{6}}=O\left(\left(\frac{\log(n)\log(\log(k))}{\log(k)}\right)^2\right).$$
\end{proof}
\subsection{Hyperbolic tangent activation}
Let us now consider the function $\tanh(x):= \frac{e^{x}-e^{-x}}{e^x + e^{-x}}$ as an activation. Since it is smooth, we should expect it to have better polynomial approximations than the ReLU. This will lead to a faster convergence rate along the CLT. An explicit expression for $\widehat{\tanh}_m$ may be difficult to find. However, one may combine the smoothness of $\tanh$ with a classical result of Hille (\cite{hille1940contributions}) in order to bound the coefficients from above. 

This calculation was done in \cite{panigrahi2019effect}, where it was shown that for the derivative $|\widehat{\tanh'}_m| \leq e^{-C\sqrt{m}}$, where $C > 0$, does not depend on $m$. We now extend this result to $\tanh$.
\begin{lemma}\label{lem:tanhcoeffs}
	Let $m \geq 0$. It holds that
	$$|\widehat{\tanh}_m| \leq e^{-C\sqrt{m}},$$
	for some absolute constant $C > 0$.
\end{lemma}
\begin{proof}
	Since $|\widehat{\tanh'}_m| \leq e^{-C\sqrt{m}}$, Hille's result (\cite[Theorem 1]{hille1940contributions}) shows that we have the point-wise equality,
	$$\tanh'(x) = \sum\limits_{m=0}^\infty\widehat{\tanh'}_mh_m(x).$$
	We now use \eqref{eq:hermitederivative}, and integrate the series, term by term, so that
	$$\tanh(x) = \sum\limits_{m=1}^\infty\frac{\widehat{\tanh'}_{m-1}}{\sqrt{m}}h_m(x).$$
	So, $\widehat{\tanh}_m = \frac{\widehat{\tanh'}_{m-1}}{\sqrt{m}}$, which proves the claim.
\end{proof}
From the lemma, we get that there is some absolute constant $C > 0$, such that $R_{\tanh}(d) \leq e^{-C\sqrt{d}}$. This allows us to prove Theorem \ref{thm:tanhactiv}.
\begin{proof}[Proof of Theorem \ref{thm:tanhactiv}]
	From Lemma \ref{lem:tanhcoeffs} along with Theorem \ref{thm:generalproc}, we get
	$$\mathcal{W}^2_2(\mathcal{P}_k\sigma, \mathcal{G}) \leq C\left(\frac{1}{k^{\frac{1}{6}}} + \exp\left(-\frac{1}{C}\sqrt{\frac{\log(k)}{\log(n)\log(\log(k))}}\right)\right).$$
	As before, the claim follows since,
	$$\frac{1}{k^\frac{1}{6}}=O\left(\exp\left(-\frac{1}{C}\sqrt{\frac{\log(k)}{\log(n)\log(\log(k))}}\right)\right).$$
\end{proof}
\bibliographystyle{plain}
\bibliography{bib}{}
\appendix
\section{Dimension-free covariance estimates for quadratic tensor powers}
When considering the polynomial $p(x) = x^2$, we can strictly improve upon Lemma \ref{lem:sigma_bound} and obtain dimension-free bounds. As noted in the proof of Theorem \ref{thm:monomctivs}, this explains Remark \ref{rmk:quadratic}.
\begin{lemma} \label{lem:improvedQuadratic}
	Suppose that $d = 2$. Then, 
	$$\|\Sigma\|_{op} \leq 1.$$
\end{lemma}
\begin{proof}
	As in the proof of Lemma \ref{lem:lowerSigmaPower}, let $v = \sum\limits_{i,j=1}^n v_{i,j}e_i\otimes e_j$, with $\sum v_{i,j}^2 = 1$. Define $q(x) = \sum\limits_{i,j=1}v_{i,j}x_ix_j$. It will suffice to bound $\mathrm{Var}(q(w))$ from above.
	Since $q$ is a quadratic polynomial, the variance decomposition \eqref{eq:nourdinDecomp} gives,
	$$\mathrm{Var}(q(w)) = \|\EE\left[\nabla q(w)\right]\|^2 + \frac{1}{2}\|\EE\left[\nabla^2 q(w)\right]\|^2.$$
	for $i \in [n]$, we have $\frac{d}{dx_i}q(w) = \sum\limits_{j=1}^n (1+\delta_{i,j})v_{i,j}w_j$. So, $\EE\left[\nabla q(w)\right] = 0$. On the other hand,
	$$\|\EE\left[\nabla q(w)\right]\|^2 = \sum\limits_{i,j=1}^n\EE\left[\frac{d^2}{dx_idx_j}q(w)\right]^2 \leq  2\sum\limits_{i=1}^nv_{i,j}^2 = 2,$$
	and the claim is proven.
\end{proof}
\end{document}